\documentclass[11pt]{amsart}

\usepackage[T1]{fontenc}
\usepackage{lmodern}
\usepackage{microtype}
\usepackage{amsmath,amssymb,amsthm,mathtools}
\usepackage[margin=1in]{geometry}
\usepackage{enumitem}
\usepackage{booktabs}
\usepackage{aliascnt}
\usepackage[colorlinks=true,linkcolor=blue,citecolor=blue,urlcolor=blue]{hyperref}
\usepackage[nameinlink,capitalise]{cleveref}

\newtheorem{theorem}{Theorem}[section]
\newaliascnt{corollary}{theorem}
\newtheorem{corollary}[corollary]{Corollary}
\aliascntresetthe{corollary}
\newaliascnt{proposition}{theorem}
\newtheorem{proposition}[proposition]{Proposition}
\aliascntresetthe{proposition}
\newaliascnt{lemma}{theorem}
\newtheorem{lemma}[lemma]{Lemma}
\aliascntresetthe{lemma}
\newaliascnt{claim}{theorem}

\aliascntresetthe{claim}
\theoremstyle{definition}
\newaliascnt{definition}{theorem}

\aliascntresetthe{definition}
\theoremstyle{remark}
\newaliascnt{remark}{theorem}
\newtheorem{remark}[remark]{Remark}
\aliascntresetthe{remark}

\crefname{theorem}{Theorem}{Theorems}
\Crefname{theorem}{Theorem}{Theorems}
\crefname{lemma}{Lemma}{Lemmas}
\Crefname{lemma}{Lemma}{Lemmas}
\crefname{proposition}{Proposition}{Propositions}
\Crefname{proposition}{Proposition}{Propositions}
\crefname{corollary}{Corollary}{Corollaries}
\Crefname{corollary}{Corollary}{Corollaries}

\numberwithin{equation}{section}

\newcommand{\1}{\mathbf 1}
\newcommand{\Jall}{\mathcal J}
\newcommand{\Proj}{\Pi_0}
\newcommand{\per}{\operatorname{per}}
\newcommand{\tr}{\operatorname{tr}}
\newcommand{\op}{\mathrm{op}}
\newcommand{\F}{\mathrm F}
\newcommand{\E}{\mathbb E}
\newcommand{\cM}{\mathcal M}
\newcommand{\abs}[1]{\left\lvert #1\right\rvert}
\newcommand{\norm}[1]{\left\lVert #1\right\rVert}

\title[Sublinear Latin rectangles and Erd\H{o}s Problem 725]{The Godsil--McKay Asymptotic for Latin Rectangles in the Sublinear Range of Erd\H{o}s Problem 725}
\author[E. Li]{Eric Li}
\dedicatory{\normalfont\normalsize Trinity College, University of Cambridge}
\date{2 August 2026}
\thanks{Email: \href{mailto:contact@ericli.com}{contact@ericli.com}.}
\subjclass[2020]{Primary 05B15; Secondary 05A16, 05C70, 15A15, 60C05}
\keywords{Latin rectangles, Erd\H{o}s Problem 725, permanents, cluster expansions, exchangeable pairs, formal verification}

\hypersetup{
  pdftitle={The Godsil--McKay Asymptotic for Latin Rectangles in the Sublinear Range of Erd\H{o}s Problem 725},
  pdfauthor={Eric Li},
  pdfsubject={Godsil--McKay asymptotic for Latin rectangles in the range k=o(n), formally verified in Lean},
  pdfkeywords={Latin rectangles, Erd\H{o}s Problem 725, permanents, cluster expansions, exchangeable pairs, formal verification}
}

\begin{document}

\begin{abstract}
Erd\H{o}s Problem~725 asks for an asymptotic formula for the number $L_{k,n}$
of ordered, labelled $k\times n$ Latin rectangles.  Godsil and McKay proved
that
\[
L_{k,n}\sim
(n!)^k\left(\frac{(n)_k}{n^k}\right)^n
\left(1-\frac{k}{n}\right)^{-n/2}e^{-k/2}
\]
for $k=o(n^{6/7})$.  We provide a partial solution to Erd\H{o}s Problem~725
by proving this asymptotic for every $k=o(n)$.  More precisely, set
\[
\widetilde A_{k,n}
=(n!)^k\left(\frac{(n)_k}{n^k}\right)^n
\exp\!\left\{\frac12\bigl[n(H_n-H_{n-k})-k\bigr]\right\}.
\]
For every $K(n)=o(n)$, uniformly for $0\le k\le K(n)$,
\[
\log\frac{L_{k,n}}{\widetilde A_{k,n}}
=O\!\left(\frac{k^2}{n^2}\right).
\]
The implied constant is absolute.  The results of this paper have been
formally verified in Lean.
\end{abstract}

\maketitle
\enlargethispage{3pt}

\section{Introduction}

Throughout, $n$ is a positive integer and $0\le k\le n$.  A $k\times n$ \emph{Latin rectangle} is an array with entries in $[n]=\{1,\dots,n\}$ such that every row is a permutation of $[n]$ and no symbol is repeated in a column.  Rows, columns, and symbols are labelled, and the rows are ordered.  We write $L_{k,n}$ for the number of such rectangles and put $L_{0,n}=1$.  We use the falling-factorial notation
\[
(x)_r=x(x-1)\cdots(x-r+1),
\qquad
(x)_0=1.
\]

Erd\H{o}s Problem~725 asks for an asymptotic formula for $L_{k,n}$.\footnote{All
results of this paper have been formally verified in Lean; see
\cref{sec:formalisation} and \cite{LiLean725}.}  We provide a partial solution
to the problem by establishing, throughout the full sublinear range,
\begin{equation}\label{eq:GM-target-intro}
L_{k,n}\sim
(n!)^k\left(\frac{(n)_k}{n^k}\right)^n
\left(1-\frac{k}{n}\right)^{-n/2}e^{-k/2}
\end{equation}
for every $k=o(n)$, uniformly on every prescribed sublinear range.  Writing
$A_{k,n}$ for the right-hand side, for every $K(n)=o(n)$ we prove uniformly
for $0\le k\le K(n)$ that
\[
\frac{L_{k,n}}{A_{k,n}}=1+O\!\left(\frac{k}{n}\right).
\]
The refined normalization $\widetilde A_{k,n}$ defined in
\eqref{eq:Atilde-def} improves both errors uniformly:
\[
\log\frac{L_{k,n}}{\widetilde A_{k,n}}
=O\!\left(\frac{k^2}{n^2}\right),
\qquad
\frac{L_{k,n}}{\widetilde A_{k,n}}
=1+O\!\left(\frac{k^2}{n^2}\right).
\]

Equivalently, a Latin rectangle is an ordered family of $k$ edge-disjoint perfect matchings of $K_{n,n}$: the left vertices are columns, the right vertices are symbols, and each row supplies one perfect matching.  Consequently, if an $m\times n$ rectangle has already been chosen, then the possible next rows are the perfect matchings of the complementary $(n-m)$-regular bipartite graph.

The asymptotic problem was initiated by Erd\H{o}s and Kaplansky
\cite{ErdosKaplansky}.  Subsequent work of Yamamoto and Stein enlarged the
range in which a growing number of rows could be treated; see
\cite{Yamamoto1951,Stein}, the historical account in
\cite[Section~1]{GodsilMcKay}, and the broader survey \cite{StonesSurvey}.
Godsil and McKay used the extension viewpoint, together with a rook-polynomial
expansion, to prove \eqref{eq:GM-target-intro} for $k=o(n^{6/7})$; see
\cite{GodsilMcKay}.  We establish it throughout the full sublinear range.

An unpublished formula stated in public lecture slides by Wormald is
attributed there to Leckey, Liebenau, and Wormald for
$k=o(n/\log^3 n)$; the same slides note that the displayed expression is
asymptotic to the Godsil--McKay expression whenever $k=o(n)$; see
\cite[slides~58--59]{WormaldSlides}.  Separately, a 2025 seminar announcement
describes joint work by Liebenau with Ledecky and Wormald, but gives no theorem
statement or parameter range; see \cite{LiebenauSeminar}.  Our comparison is
solely with these public statements, and we make no assertion concerning
unpublished results.  Relative to the range stated in Wormald's slides,
\cref{thm:main} extends $o(n/\log^3 n)$ to arbitrary prescribed sublinear
ranges.

For $j\ge1$, let
\[
H_j=\sum_{r=1}^j\frac1r,
\qquad H_0=0,
\]
and, for $0\le k<n$, define
\begin{align}
\widetilde A_{k,n}
&=(n!)^k\left(\frac{(n)_k}{n^k}\right)^n
  \exp\!\left\{\frac12\bigl[n(H_n-H_{n-k})-k\bigr]\right\},
  \label{eq:Atilde-def}\\
A_{k,n}
&=(n!)^k\left(\frac{(n)_k}{n^k}\right)^n
  \left(1-\frac{k}{n}\right)^{-n/2}e^{-k/2}.
  \label{eq:A-def}
\end{align}

Whenever a range function is used below, it means a map
$K:\mathbb N\to\mathbb Z_{\ge0}$ satisfying $K(n)=o(n)$.  In particular,
$K(n)<n$ for all sufficiently large $n$.

\begin{theorem}[Quantitative sublinear enumeration]\label{thm:main}
There is an absolute constant $C_{\mathrm{main}}<\infty$ with the following
property.  For every such function $K$,
there is an index $n_0=n_0(K)$ such that, whenever $n\ge n_0$ and
$0\le k\le K(n)$,
\begin{equation}\label{eq:main-tilde}
\left|\log\frac{L_{k,n}}{\widetilde A_{k,n}}\right|
\le C_{\mathrm{main}}\frac{k^2}{n^2}.
\end{equation}
Consequently, after increasing $C_{\mathrm{main}}$ if necessary,
\begin{equation}\label{eq:main-GM}
\left|\log\frac{L_{k,n}}{A_{k,n}}\right|
\le C_{\mathrm{main}}\left(\frac{k}{n}+\frac{k^2}{n^2}\right)
\le 2C_{\mathrm{main}}\frac{k}{n}
\end{equation}
throughout the same range.
\end{theorem}

\begin{corollary}[Relative forms]\label{cor:sublinear}
For every such function $K$, uniformly for $0\le k\le K(n)$,
\begin{align}
\frac{L_{k,n}}{\widetilde A_{k,n}}
&=1+O\!\left(\frac{k^2}{n^2}\right),
\label{eq:sublinear-relative-tilde}\\
\frac{L_{k,n}}{A_{k,n}}
&=1+O\!\left(\frac{k}{n}\right).
\label{eq:sublinear-relative}
\end{align}
In particular, if $k=o(n)$, then
\begin{equation}\label{eq:sublinear-asymptotic}
L_{k,n}\sim A_{k,n}.
\end{equation}
The constants in the two $O$-terms are absolute; the index from which the
estimates hold may depend on the prescribed function $K$.
\end{corollary}

\begin{remark}[Dependence of constants and thresholds]\label{rem:dependencies}
The uniformity used later is summarized in the following table.
\begin{center}
\small
\begin{tabular}{@{}ll@{}}
\toprule
quantity & dependence \\
\midrule
$C_{\mathrm{perm}}(C_0)$ & $C_0$ only \\
permanent-theorem threshold & $C_0$ and the prescribed sequence $\eta_\bullet$ \\
$c_{\mathrm{ext}}$ & absolute, because the Latin application has $C_0=1$ \\
one-row threshold & the prescribed range function $K$ \\
$C_{\mathrm{main}}$ & absolute \\
main-theorem threshold & the prescribed range function $K$ \\
\bottomrule
\end{tabular}
\end{center}
An ``absolute'' constant is independent of $K$; the index from which an
estimate holds may depend on $K$ as indicated.
\end{remark}

For fixed $k$, the normalization in \eqref{eq:Atilde-def} captures one
additional order in the small-height expansion beyond the normalization
stated in Wormald's public slides.

\begin{proposition}[Comparison with the normalization stated in Wormald's public slides]\label{prop:LLW-comparison}
Define
\begin{equation}\label{eq:LLW-def}
\operatorname{LLW}_{k,n}
=\frac{(n!)^k((n)_k)^{2n}}{e^{k/2}(n^2)_{kn}}.
\end{equation}
For $0\le k<n$,
\begin{align}
\log\frac{\operatorname{LLW}_{k,n}}{\widetilde A_{k,n}}
={}&
n\sum_{j=0}^{k-1}\log\left(1-\frac{j}{n}\right)
-\sum_{r=0}^{kn-1}\log\left(1-\frac{r}{n^2}\right) \notag\\
&-\frac12\sum_{j=0}^{k-1}\frac{n}{n-j}.
\label{eq:LLW-exact-comparison}
\end{align}
For every fixed $k$,
\begin{equation}\label{eq:LLW-fixed-k}
\log\frac{\operatorname{LLW}_{k,n}}{\widetilde A_{k,n}}
=-\frac{k}{3n}-\frac{k(k+1)}{12n^2}+O_k(n^{-3}).
\end{equation}
\end{proposition}

\begin{proof}
Substituting \eqref{eq:Atilde-def} and \eqref{eq:LLW-def}, cancelling
$(n!)^k$ and $e^{-k/2}$, and extracting the powers of $n$ from the two
falling factorials gives \eqref{eq:LLW-exact-comparison}.  For fixed $k$, put $S_q=\sum_{j=0}^{k-1}j^q$ and
$T_q=\sum_{r=0}^{kn-1}r^q$.  Expanding the three finite sums gives
\begin{align*}
n\sum_{j<k}\log(1-j/n)
&=-S_1-\frac{S_2}{2n}-\frac{S_3}{3n^2}+O_k(n^{-3}),\\
-\sum_{r<kn}\log(1-r/n^2)
&=\frac{T_1}{n^2}+\frac{T_2}{2n^4}
  +\frac{T_3}{3n^6}+O_k(n^{-3}),\\
-\frac12\sum_{j<k}\frac{n}{n-j}
&=-\frac{k}{2}-\frac{S_1}{2n}-\frac{S_2}{2n^2}
  +O_k(n^{-3}).
\end{align*}
Using the standard power-sum formulae for $S_q$ and $T_q$, the constant
terms cancel and the coefficients of $n^{-1}$ and $n^{-2}$ are respectively
$-k/3$ and $-k(k+1)/12$.  This proves \eqref{eq:LLW-fixed-k}.
\end{proof}

\begin{remark}[Checks at small height]\label{rem:small-height-checks}
At $k=0$ and $k=1$ one has
\[
L_{0,n}=\widetilde A_{0,n}=1,
\qquad
L_{1,n}=\widetilde A_{1,n}=n!
\]
exactly.  At $k=2$, write $D_n$ for the number of derangements of $[n]$.
Then $L_{2,n}=n!D_n$ and
$D_n=(n!/e)(1+O((n+1)!^{-1}))$.  Direct expansion of
\eqref{eq:Atilde-def} gives
\begin{equation}\label{eq:k2-check}
\log\frac{L_{2,n}}{\widetilde A_{2,n}}
=-\frac1{6n^2}-\frac1{4n^3}-\frac3{10n^4}+O(n^{-5}).
\end{equation}
Thus the $n^{-2}$ scale in \eqref{eq:main-tilde} already occurs at height
two.
\end{remark}

The first ingredient is a permanent approximation adapted to matrices whose deviations from the all-ones matrix have small row and column $\ell_1$-norms.  If $E$ has zero row and column sums, bounded entries, and maximum row or column $\ell_1$-norm $s(E)=o(n)$, then
\begin{equation}\label{eq:permanent-intro}
\log\frac{\per(\Jall+E)}{n!}
=-\frac12\log\det\!\left(I-\frac{E^{\mathsf T}E}{n^2}\right)
+O\!\left(\frac{s(E)}{n^2}\right).
\end{equation}
The proof uses an exact Gamma representation for the permanent, a convergent hard-core matching expansion, and a complete separation of the cycle sector from all connected incidence graphs of positive excess.  Under the normalization $A=(\Jall+E)/n$, the determinant in
\eqref{eq:permanent-intro} is exactly McCullagh's leading determinant
\cite{McCullagh}.  Except for the standard Penrose tree--graph inequality quoted in \eqref{eq:tree-graph}, all ingredients needed for the class-uniform remainder of order $s(E)/n^2$ are proved here.

The second ingredient is a fixed-anchor sampling of the standard incomplete two-column cycle/path switch for Latin rectangles.  The switching gives a linear exchangeable-pair regression for every two-column codegree and the upper bound
\[
\operatorname{Var}Z
\le \frac{m(m-1)(n-m)}{(n-1)^2}.
\]
An exact codegree--spectral identity then yields
\[
\E\tr\bigl((C^{\mathsf T}C)^2\bigr)=O(nm^2),
\]
where $C$ is the centred $m$-row column--symbol incidence matrix.

For a residual graph at height $m=o(n)$, the deviation matrix in \eqref{eq:permanent-intro} has line sum $2m$.  Its determinant has a deterministic first trace $m/(n-m)$, while the switching estimate shows that the expected nonlinear trace-log remainder is $O(m^2/n^3)$.  We obtain
\[
\frac{L_{m+1,n}}{L_{m,n}}
=n!\left(1-\frac mn\right)^n
 \exp\!\left(\frac{m}{2(n-m)}\right)
 \left[1+O\!\left(\frac{m}{n^2}\right)\right]
\]
uniformly for $m$ in any sublinear range.  Multiplication over $m=0,\dots,k-1$ proves \cref{thm:main}.

\subsection{Formalisation}\label{sec:formalisation}

Every mathematical result proved in this paper has been formalised in Lean
using mathlib.  The development verifies the quantitative sublinear
enumeration theorem and its relative forms together with the supporting
permanent, cluster-expansion, switching, spectral, extension, and finite-height
results.  The verification is unconditional relative to Lean's standard
classical foundations: every intermediate interface introduced by the
formalisation is discharged internally, and the development contains no
admitted proof or project-defined axiom.  The source, build instructions, and
kernel-audit procedure are available in the accompanying repository
\cite{LiLean725}.

\section{A permanent theorem for small row and column line norms}\label{sec:permanent}

Let $\Jall$ denote the $n\times n$ all-ones matrix.  For a real $n\times n$ matrix $E$, define
\begin{equation}\label{eq:line-norm}
s(E)=\max\left\{
\max_i\sum_j\abs{E_{ij}},
\max_j\sum_i\abs{E_{ij}}
\right\}.
\end{equation}

\begin{theorem}[Small-line-norm permanent approximation]\label{thm:small-line-permanent}
For every $0\le C_0<\infty$ there is a constant
$C_{\mathrm{perm}}=C_{\mathrm{perm}}(C_0)<\infty$ with the following
property.  For every sequence $\eta_\bullet=(\eta_n)_{n\ge1}$ satisfying
$\eta_n\ge0$ and $\eta_n\to0$, there is an index
$n_0=n_0(C_0,\eta_\bullet)$ such that, for every $n\ge n_0$ and every real
$n\times n$ matrix $E$ satisfying
\begin{equation}\label{eq:permanent-hypotheses}
E\1=E^{\mathsf T}\1=0,
\qquad
\Jall+E\ge0\ \text{entrywise},
\qquad
\max_{i,j}\abs{E_{ij}}\le C_0,
\qquad
s(E)\le \eta_n n,
\end{equation}
one has
\begin{equation}\label{eq:small-line-permanent}
\left|
\log\frac{\per(\Jall+E)}{n!}
+\frac12\log\det\!\left(I-\frac{E^{\mathsf T}E}{n^2}\right)
\right|
\le C_{\mathrm{perm}}\frac{s(E)}{n^2}.
\end{equation}
The constant $C_{\mathrm{perm}}$ is independent of the sequence
$\eta_\bullet$, of $n$, and of the particular matrix $E$; the threshold
$n_0$ may depend on $C_0$ and $\eta_\bullet$.
\end{theorem}

\begin{remark}[Comparison with McCullagh's approximation]\label{rem:permanent-comparison}
Set
\[
A=\frac{\Jall+E}{n},
\qquad
\Proj=\frac{\Jall}{n}.
\]
Because $E\1=E^{\mathsf T}\1=0$, one has
\[
A=\Proj+\frac En,
\qquad
I+\Proj-A^{\mathsf T}A
=I-\frac{E^{\mathsf T}E}{n^2}.
\]
Thus the determinant in \cref{thm:small-line-permanent} is exactly
McCullagh's leading determinant under the present normalization
\cite{McCullagh}.  The connected pair-of-partitions representation below is
also closely related to McCullagh's partition expansion, and his
pair-partition contribution corresponds to our incidence-excess-zero cycle
sector.

Under these bounded-entry and small-line-norm assumptions, the normalized
matrices have bounded entry deviations and $o(1)$ spectral deviation; away
from boundary zeros they therefore lie in McCullagh's moderate-deviation
regime, which gives an $O(n^{-1})$ approximation along such sequences.
McCullagh also discusses scalable nonnegative matrices containing zeros.  The
new assertion required here is the sharper class-uniform remainder
$O_{C_0}(s(E)/n^2)$, together with a proof that handles the present zero
entries directly.  Except for the standard Penrose tree--graph inequality in
\eqref{eq:tree-graph}, the ingredients used to obtain that class-uniform
remainder are proved below: the Gamma representation, the absolute control of
all positive-incidence-excess contractions, and the Gamma averaging.
\end{remark}

We write $s=s(E)$ and $\Lambda_0=\max\{1,C_0\}$.  The case $s=0$ is
immediate, so we assume $s>0$.  The matrix $(\Jall+E)/n$ is nonnegative and
doubly stochastic.  Its support satisfies Hall's condition.  If $I$ is a set of rows,
$N(I)$ is the set of columns meeting a positive entry in those rows, and
$P=(\Jall+E)/n$, then
\[
|I|
=\sum_{i\in I}\sum_j P_{ij}
=\sum_{j\in N(I)}\sum_{i\in I}P_{ij}
\le\sum_{j\in N(I)}\sum_iP_{ij}
=|N(I)|.
\]
Hence the support contains a perfect matching and
$\per(\Jall+E)>0$, so the real logarithm in
\eqref{eq:small-line-permanent} is defined.  Moreover,
\[
\norm{E}_{\op}\le
\sqrt{\norm{E}_1\norm{E}_\infty}\le s=o(n),
\]
so every eigenvalue of $E^{\mathsf T}E/n^2$ is strictly less than one for
all sufficiently large $n$.  The determinant in
\eqref{eq:small-line-permanent} is therefore positive.

\subsection{An exact Gamma representation}

Let $\cM$ be the family of all matchings in $K_{n,n}$, including the empty matching, and put
\begin{equation}\label{eq:ZE-def}
Z_E(z)=\sum_{M\in\cM}z^{\abs M}\prod_{(i,j)\in M}E_{ij}.
\end{equation}

\begin{lemma}[Gamma representation]\label{lem:gamma-representation}
Let $T$ have the Gamma distribution with shape $n+1$ and rate $1$, so that its density is $t^ne^{-t}/n!$ on $(0,\infty)$.  Then
\begin{equation}\label{eq:gamma-representation}
\frac{\per(\Jall+E)}{n!}=\E Z_E(T^{-1}).
\end{equation}
\end{lemma}

\begin{proof}
If $M$ is a matching of size $r$, choosing the entries of $E$ precisely on $M$ leaves $(n-r)!$ ways to complete a permutation using entries of $\Jall$.  Therefore
\[
\per(\Jall+E)
=\sum_{M\in\cM}(n-\abs M)!\prod_{e\in M}E_e.
\]
For $0\le r\le n$,
\[
\E T^{-r}
=\frac1{n!}\int_0^\infty t^{n-r}e^{-t}\,dt
=\frac{(n-r)!}{n!}.
\]
Substitution gives \eqref{eq:gamma-representation}.
\end{proof}

\subsection{The matching cluster expansion}

Regard the cells $(i,j)\in[n]^2$ as polymers.  Two polymers are incompatible if they share a row or a column; a polymer is incompatible with itself.  Give cell $(i,j)$ activity $w_{ij}$.  The corresponding hard-core partition function is
\[
Z(w)=\sum_{M\in\cM}\prod_{e\in M}w_e.
\]
For an ordered $r$-tuple $\boldsymbol e=(e_1,\dots,e_r)$, let
\[
\phi^T(\boldsymbol e)
=\sum_{G\in\mathcal C_r}
  \prod_{\{a,b\}\in E(G)}
  \bigl(-\1_{\{e_a\text{ incompatible with }e_b\}}\bigr),
\]
where $\mathcal C_r$ is the set of connected simple graphs on $[r]$; for $r=1$ put $\phi^T(e_1)=1$.

The incompatibility graph in this definition is a graph on the tuple
positions $[r]$.  In particular, two positions carrying the same cell are
adjacent because a polymer is incompatible with itself.  We use the standard
Penrose tree-graph inequality
\begin{equation}\label{eq:tree-graph}
\abs{\phi^T(e_1,\dots,e_r)}
\le
\sum_{\tau\in\mathcal T_r}
\prod_{\{a,b\}\in E(\tau)}
\1_{\{e_a\text{ incompatible with }e_b\}},
\end{equation}
where $\mathcal T_r$ is the set of labelled trees on $[r]$.  This is the
hard-core specialization of the Penrose identity; see
\cite[Proposition~5 and equations~(4.3)--(4.5)]{FernandezProcacci}.  The
following elementary consequence records the precise signed/complex version
used here.

\begin{lemma}[Absolute cluster bound]\label{lem:cluster-bound}
Put
\[
W=\sum_e\abs{w_e},
\qquad
\Delta=\max_e\sum_{f:\,f\text{ incompatible with }e}\abs{w_f}.
\]
If $e\Delta<1$, then
\begin{equation}\label{eq:cluster-series}
\log Z(w)
=\sum_{r\ge1}\frac1{r!}
 \sum_{e_1,\dots,e_r}
 \phi^T(e_1,\dots,e_r)\prod_{a=1}^rw_{e_a},
\end{equation}
and the series converges absolutely.  Moreover, if $R\ge1$ and $e\Delta\le1/2$, then
\begin{equation}\label{eq:cluster-tail}
\sum_{r>R}\frac1{r!}
 \sum_{e_1,\dots,e_r}
 \abs{\phi^T(e_1,\dots,e_r)}
 \prod_{a=1}^r\abs{w_{e_a}}
\le 4W(e\Delta)^R.
\end{equation}
\end{lemma}

\begin{proof}
For $r=1$ the absolute cluster contribution is $W$.  For $r\ge2$, fix a tree, choose a root polymer in total weight $W$, and sum the remaining vertices away from the root.  Every tree edge costs at most a factor $\Delta$.  By \eqref{eq:tree-graph} and Cayley's formula,
\begin{equation}\label{eq:tree-sum-bound}
\frac1{r!}\sum_{e_1,\dots,e_r}
\abs{\phi^T(e_1,\dots,e_r)}
\prod_{a=1}^r\abs{w_{e_a}}
\le \frac{r^{r-2}}{r!}W\Delta^{r-1}.
\end{equation}
The right side is summable when $e\Delta<1$, because $r!\ge(r/e)^r$.
For completeness, the formal connected-graph identity can be seen directly.
Set
\[
f(e_a,e_b)=-\1_{\{e_a\text{ incompatible with }e_b\}}.
\]
Because self-incompatibility makes repeated cells incompatible, the hard-core
partition function has the exact ordered-tuple representation
\[
Z(\lambda w)
=1+\sum_{r\ge1}\frac{\lambda^r}{r!}
 \sum_{e_1,\dots,e_r}
 \prod_{a=1}^r w_{e_a}
 \prod_{1\le a<b\le r}\bigl(1+f(e_a,e_b)\bigr).
\]
Expanding the last product sums over all simple graphs on the tuple positions.
Every such graph decomposes uniquely into its connected components, and both
the edge factor and the activity product factor over those components.  The
labelled exponential formula therefore gives, coefficient by coefficient,
$Z(\lambda w)=\exp F(\lambda)$, where
\[
F(\lambda)=\sum_{r\ge1}\frac{\lambda^r}{r!}
 \sum_{e_1,\dots,e_r}\phi^T(e_1,\dots,e_r)\prod_{a=1}^rw_{e_a}.
\]
For complex $\lambda$, the preceding absolute estimate shows that the series
defining $F$ converges normally on every compact subset of
$\abs\lambda<(e\Delta)^{-1}$.  Thus the formal identity is an analytic
identity near zero, and the identity theorem extends it throughout that disk.
In particular, for real activities and $\lambda=1$, $F(1)$ is real and
$Z(w)=e^{F(1)}>0$; the logarithm in \eqref{eq:cluster-series} is therefore
the real branch obtained continuously from $Z(0)=1$.

For $r\ge2$, the elementary bound $r^{r-2}/r!\le e^{r-1}$ follows from
$r!\ge(r/e)^r$ and $r^2\ge e$.  Hence, when $e\Delta\le1/2$,
\[
\sum_{r>R}\frac{r^{r-2}}{r!}W\Delta^{r-1}
\le W\sum_{r>R}(e\Delta)^{r-1}
\le2W(e\Delta)^R,
\]
which is stronger than \eqref{eq:cluster-tail}.
\end{proof}

For the activities $w_{ij}=E_{ij}/t$, with $t\ge n/2$, one has
\begin{equation}\label{eq:WDelta}
W\le\frac{ns}{t}\le2s,
\qquad
\Delta\le\frac{2s}{t}\le\frac{4s}{n}=o(1).
\end{equation}
Thus the cluster expansion is absolutely convergent, uniformly for $t\ge n/2$, once $n$ is sufficiently large.

\subsection{A connected partition formula}

Let $\Pi_r$ be the lattice of partitions of $[r]$, with least and greatest elements $\widehat0$ and $\widehat1$.  Put
\begin{equation}\label{eq:mobius-weight}
\mu_{\Pi}(\rho)=\prod_{B\in\rho}(-1)^{\abs B-1}(\abs B-1)!.
\end{equation}
For $\rho,\sigma\in\Pi_r$, define
\begin{equation}\label{eq:contraction}
E^\rho_\sigma
=\sum_{\alpha:\rho\to[n]}
 \sum_{\beta:\sigma\to[n]}
 \prod_{\ell=1}^r
 E_{\alpha(B_\rho(\ell)),\,\beta(B_\sigma(\ell))},
\end{equation}
where $B_\rho(\ell)$ is the block of $\rho$ containing $\ell$.

\begin{lemma}[Connected partition coefficients]\label{lem:connected-partition}
As a formal power series at the origin,
\begin{equation}\label{eq:connected-partition}
\log Z_E(z)
=\sum_{r\ge1}\frac{z^r}{r!}
 \sum_{\substack{\rho,\sigma\in\Pi_r\\
                   \rho\vee\sigma=\widehat1}}
 \mu_{\Pi}(\rho)\mu_{\Pi}(\sigma)E^\rho_\sigma.
\end{equation}
Equivalently, the coefficient of $z^r$ is the displayed finite sum for every
$r\ge1$.  If either $\rho$ or $\sigma$ has a singleton block, then
$E^\rho_\sigma=0$.
\end{lemma}

\begin{proof}
For a tuple $\boldsymbol i=(i_1,\dots,i_r)$, let $\ker\boldsymbol i$ be its
equality partition.  M\"obius inversion on $\Pi_r$ gives
\[
\1_{\{i_1,\dots,i_r\text{ distinct}\}}
=\sum_{\rho\le\ker\boldsymbol i}\mu_{\Pi}(\rho),
\]
and similarly for the column indices.  Hence, as a polynomial identity,
\begin{equation}\label{eq:ZE-partitions}
Z_E(z)
=\sum_{r\ge0}\frac{z^r}{r!}
  \sum_{\rho,\sigma\in\Pi_r}
  \mu_{\Pi}(\rho)\mu_{\Pi}(\sigma)E^\rho_\sigma,
\end{equation}
where coefficients with $r>n$ vanish.

For a nonempty finite label set $B$, let $\Pi(B)$ be its partition lattice and
define the connected weight
\begin{equation}\label{eq:bB-def}
b(B)=
\sum_{\substack{\rho,\sigma\in\Pi(B)\\
                  \rho\vee\sigma=\{B\}}}
 \mu_{\Pi}(\rho)\mu_{\Pi}(\sigma)E^\rho_\sigma.
\end{equation}
Here $E^\rho_\sigma$ is defined as in \eqref{eq:contraction}, with the edge
labels drawn from $B$.

Given arbitrary $\rho,\sigma\in\Pi_r$, put $\tau=\rho\vee\sigma$.  The blocks
of $\tau$ are exactly the connected components of the bipartite incidence
multigraph of $(\rho,\sigma)$.  Restriction to each $B\in\tau$ produces a
connected pair $(\rho|_B,\sigma|_B)$, and conversely a choice of one connected
pair on each block $B\in\tau$ assembles uniquely into $(\rho,\sigma)$.  The
M\"obius factors satisfy
\[
\mu_{\Pi}(\rho)\mu_{\Pi}(\sigma)
=\prod_{B\in\tau}\mu_{\Pi}(\rho|_B)\mu_{\Pi}(\sigma|_B).
\]
The contraction also factors:
\begin{equation}\label{eq:contraction-factorization}
E^\rho_\sigma
=\prod_{B\in\tau}E^{\rho|_B}_{\sigma|_B}.
\end{equation}
Indeed, the row- and column-label maps attached to different components have
disjoint domains, so their sums are independent.  Numerical labels assigned
to different components are allowed to coincide; after M\"obius inversion
there is no cross-component distinctness condition that could couple these
sums.

It follows that the coefficient in \eqref{eq:ZE-partitions} is
\begin{equation}\label{eq:moment-connected-factorization}
a([r])=\sum_{\tau\in\Pi_r}\prod_{B\in\tau}b(B),
\end{equation}
where $a([0])=1$.  The labelled moment--cumulant identity, equivalently the
exponential formula, now gives
\[
\sum_{r\ge0}a([r])\frac{z^r}{r!}
=
\exp\!\left(\sum_{r\ge1}b([r])\frac{z^r}{r!}\right).
\]
The left side is $Z_E(z)$ by \eqref{eq:ZE-partitions}; taking the formal
logarithm proves \eqref{eq:connected-partition} coefficient by coefficient.

If $\{\ell\}$ is a singleton block of $\rho$, then the row label associated
with that block occurs in exactly one matrix factor.  Summing that label gives
a column sum of $E$, which is zero.  The argument for a singleton block of
$\sigma$ is identical.
\end{proof}

\begin{lemma}[First cluster coefficients]\label{lem:first-cluster-coefficients}
Under the zero row- and column-sum assumptions,
\begin{align}
[z]\log Z_E(z)&=0,\label{eq:cluster-coeff-1}\\
[z^2]\log Z_E(z)&=\frac12\tr(E^{\mathsf T}E),\label{eq:cluster-coeff-2}\\
[z^3]\log Z_E(z)&=\frac23\sum_{i,j}E_{ij}^3.\label{eq:cluster-coeff-3}
\end{align}
\end{lemma}

\begin{proof}
The degree-one term vanishes because its row and column partitions are both
singletons.  At degree two, the only surviving connected pair has both
partitions equal to the unique pair partition; its two M\"obius weights are
$-1$, and division by $2!$ gives
$\tfrac12\sum_{i,j}E_{ij}^2=\tfrac12\tr(E^{\mathsf T}E)$.
At degree three, absence of singleton blocks forces both partitions to be the
one-block partition.  Each M\"obius weight is $2$, the contraction is
$\sum_{i,j}E_{ij}^3$, and the factor $4/3!$ is $2/3$.
\end{proof}

The incidence multigraph of a surviving pair $(\rho,\sigma)$ is connected and has minimum degree at least two.  Write
\[
r=\text{number of edges},
\qquad
v=\abs\rho+\abs\sigma,
\qquad
g=r-v.
\]
We call $g$ the \emph{incidence excess}.  With this convention a connected cycle has incidence excess zero; its cyclomatic number is $g+1$.  In particular, $g\ge0$.

\subsection{The zero-excess cycle sector}

\begin{lemma}[Cycle resummation]\label{lem:cycle-resummation}
The generating series of the $g=0$ contributions in \eqref{eq:connected-partition} is
\begin{equation}\label{eq:cycle-resummation}
\sum_{\ell\ge1}\frac{z^{2\ell}}{2\ell}
\tr\bigl((E^{\mathsf T}E)^\ell\bigr)
=-\frac12\log\det(I-z^2E^{\mathsf T}E),
\end{equation}
whenever $\abs z\norm{E}_{\op}<1$.
\end{lemma}

\begin{proof}
If $g=0$, the connected incidence multigraph has as many edges as vertices and minimum degree at least two.  Every vertex therefore has degree two, so the graph is an alternating cycle.  Thus $r=2\ell$, and both $\rho$ and $\sigma$ are pair partitions.  Here
$\ell=1$ is the two-vertex doubled-edge cycle in the incidence multigraph;
it supplies the leading term $\tfrac12\tr(E^{\mathsf T}E)$.  The M\"obius
weights of the two pair partitions multiply to one.

The number of ordered pairs of pair partitions of $[2\ell]$ whose union is one alternating cycle is $(2\ell)!/(2\ell)$.  To see this without an orientation convention, start from a linear ordering
$(e_1,\dots,e_{2\ell})$ of the edge labels and set
\[
\rho=\bigl\{\{e_1,e_2\},\{e_3,e_4\},\dots,\{e_{2\ell-1},e_{2\ell}\}\bigr\},
\]
\[
\sigma=\bigl\{\{e_2,e_3\},\{e_4,e_5\},\dots,
                 \{e_{2\ell},e_1\}\bigr\}.
\]
This produces an ordered pair whose incidence graph is one cycle.
Conversely, from an ordered pair $(\rho,\sigma)$ and a choice of the first
edge $e_1$, the sequence is forced by alternately taking the $\rho$-partner
and the $\sigma$-partner.  Thus every ordered pair has exactly $2\ell$
preimages among the $(2\ell)!$ linear orderings.  All such pairs are
isomorphic by a relabelling of the edge labels, and their contractions are
therefore equal.  For the canonical pair the contraction is
\[
\sum_{j_1,\dots,j_\ell}
\sum_{i_1,\dots,i_\ell}
\prod_{q=1}^\ell E_{i_qj_q}E_{i_qj_{q+1}},
\qquad j_{\ell+1}=j_1,
\]
which is $\tr((E^{\mathsf T}E)^\ell)$.  Division by $(2\ell)!$ gives the coefficient $1/(2\ell)$.  Summing the matrix logarithm series proves \eqref{eq:cycle-resummation}.
\end{proof}

\subsection{Positive-excess terms}

We first bound a fixed contraction.

\begin{lemma}[Spanning-tree contraction bound]\label{lem:contraction-bound}
Let $(\rho,\sigma)$ have connected incidence multigraph, with $r$ edges, $v$ vertices, and excess $g=r-v$.  Then
\begin{equation}\label{eq:contraction-bound}
\abs{E^\rho_\sigma}
\le n\,s^{v-1}\Lambda_0^{g+1}.
\end{equation}
\end{lemma}

\begin{proof}
Choose a spanning tree of the incidence multigraph.  Bound the $r-(v-1)=g+1$ non-tree matrix factors by $\Lambda_0$.  For the remaining tree product, choose the root label in $n$ ways and eliminate leaves successively.  Each leaf-label sum is bounded by a row or column $\ell_1$-norm of $E$, hence by $s$.  There are $v-1$ nonroot vertices.
\end{proof}

We also need a bound for the total M\"obius weight at fixed excess.

\begin{lemma}[Partition-weight bound]\label{lem:partition-weight}
For $r\ge1$ and $g\ge0$,
\begin{equation}\label{eq:partition-weight}
\frac1{r!}
\sum_{\substack{\rho,\sigma\in\Pi_r\\
                 \rho,\sigma\text{ have no singleton blocks}\\
                 r-\abs\rho-\abs\sigma=g}}
\abs{\mu_{\Pi}(\rho)\mu_{\Pi}(\sigma)}
\le \frac{(4r^3)^g}{(2g)!}.
\end{equation}
\end{lemma}

\begin{proof}
For a partition $\rho$ with $p$ blocks and no singleton blocks, put $a=r-2p$.  Thus $a\ge0$ and has the same parity as $r$.  The exponential formula for a set of $p$ unordered blocks gives
\begin{align}
&\sum_{\substack{\rho\in\Pi_r\\r-2\abs\rho=a\\
                 \rho\text{ has no singleton blocks}}}
\abs{\mu_{\Pi}(\rho)} \notag\\
&\qquad=\frac{r!}{p!}[u^a]
 \left(\sum_{t\ge0}\frac{u^t}{t+2}\right)^p \notag\\
&\qquad=\frac{r!}{p!}
 \sum_{\substack{t_1+\cdots+t_p=a\\t_i\ge0}}
 \prod_{i=1}^p\frac1{t_i+2}.
\label{eq:one-partition-weight}
\end{align}
Indeed, a block of size $2+t$ contributes the M\"obius magnitude
$(1+t)!$, while its labelled-set exponential generating-function factor
is $u^t/(2+t)!$; their product is $u^t/(2+t)$.  Therefore
\begin{equation}\label{eq:one-partition-bound}
\sum_{\substack{\rho\in\Pi_r\\r-2\abs\rho=a\\
                 \rho\text{ has no singleton blocks}}}
\abs{\mu_{\Pi}(\rho)}
\le
\frac{r!}{2^pp!}\binom{a+p-1}{a}
\le
\frac{r!}{2^pp!}\frac{r^a}{a!}.
\end{equation}

For a pair of defects $a,b$, the excess condition is $a+b=2g$, and if $q=(r-b)/2$, then $p+q=r-g$.  Moreover,
\[
\frac{r!}{2^{p+q}p!q!}
=\frac{(r)_g}{2^{r-g}}\binom{r-g}{p}
\le r^g.
\]
Using \eqref{eq:one-partition-bound} twice and summing over $a+b=2g$ gives
\[
\frac1{r!}\sum\abs{\mu_{\Pi}(\rho)\mu_{\Pi}(\sigma)}
\le r^{3g}\sum_{a+b=2g}\frac1{a!b!}
=\frac{(4r^3)^g}{(2g)!}.
\]
\end{proof}

\paragraph{Truncation notation.}
For $t\ge n/2$, let $\mathcal E_{r,g}(t)$ denote the total absolute value,
after multiplication by $t^{-r}$, of all degree-$r$ connected-partition
terms having incidence excess $g$.  By
\cref{lem:contraction-bound,lem:partition-weight}, if the incidence graph is
connected and has no degree-one vertex, then
\begin{equation}\label{eq:fixed-rg-bound}
\mathcal E_{r,g}(t)
\le
2\Lambda_0\,
\theta^{\,r-g-1}
\frac{(8\Lambda_0 r^3/n)^g}{(2g)!},
\qquad
\theta=\frac{s}{t}.
\end{equation}
Indeed, the incidence graph has $v=r-g$ vertices and
\[
t^{-r}ns^{v-1}\Lambda_0^{g+1}
=
\theta^{v-1}\Lambda_0^{g+1}\frac{n}{t^{g+1}},
\]
while $n/t^{g+1}\le 2^{g+1}n^{-g}$.

\begin{lemma}[Uniform truncation parameters]\label{lem:truncation-parameters}
Fix a sequence $\eta_n\ge0$ with $\eta_n\to0$.  For every sufficiently large $n$, uniformly
over all $E$ with $0<s(E)\le\eta_n n$, put
\begin{equation}\label{eq:Rtr-def}
q=\frac{4es}{n},
\qquad
R_{\mathrm{tr}}
=
\max\left\{
3,\left\lceil\frac{4\log n}{\log(1/q)}\right\rceil
\right\}.
\end{equation}
Then
\begin{equation}\label{eq:Rtr-properties}
q\le\frac14,
\qquad
R_{\mathrm{tr}}\le 3+\frac{4\log n}{\log4},
\qquad
q^{R_{\mathrm{tr}}}\le n^{-4}.
\end{equation}
In particular, uniformly over the class,
\begin{equation}\label{eq:Rtr-growth}
R_{\mathrm{tr}}=O(\log n),
\qquad
\frac{R_{\mathrm{tr}}^7}{n}\longrightarrow0.
\end{equation}
\end{lemma}

\begin{proof}
Since $s/n\le\eta_n$, one has $q\le4e\eta_n\to0$, and hence $q\le1/4$
for all sufficiently large $n$.  This gives the upper bound for
$R_{\mathrm{tr}}$ in \eqref{eq:Rtr-properties}.

It remains to verify the power bound in the two cases in
\eqref{eq:Rtr-def}.  If the maximum selects the ceiling term, then
$R_{\mathrm{tr}}\ge 4\log n/\log(1/q)$, so
$q^{R_{\mathrm{tr}}}\le n^{-4}$.  If the maximum selects $3$, then
$4\log n/\log(1/q)\le3$, which is equivalent to
$q^3\le n^{-4}$.  This proves \eqref{eq:Rtr-properties}.
The growth assertions follow immediately.
\end{proof}

\begin{lemma}[Excess-one bound]\label{lem:excess-one-bound}
For every $0\le C_0<\infty$ there is a constant
$C_{\mathrm{one}}=C_{\mathrm{one}}(C_0)$ such that, for all sufficiently
large $n$ in the setting of \cref{lem:truncation-parameters} and every
$t\ge n/2$,
\begin{equation}\label{eq:excess-one-bound}
\sum_{r\le R_{\mathrm{tr}}}\mathcal E_{r,1}(t)
\le C_{\mathrm{one}}\frac{\theta}{n}.
\end{equation}
\end{lemma}

\begin{proof}
For $g=1$, connectedness and minimum degree two imply $r\ge3$.  From
\eqref{eq:fixed-rg-bound},
\[
\mathcal E_{r,1}(t)
\le \frac{8\Lambda_0^2}{n}r^3\theta^{r-2}.
\]
Moreover,
\[
\theta=\frac{s}{t}\le\frac{2s}{n}=\frac{q}{2e}\le\frac1{8e}.
\]
Therefore
\[
\sum_{r\le R_{\mathrm{tr}}}\mathcal E_{r,1}(t)
\le
\frac{8\Lambda_0^2\theta}{n}
\sum_{u\ge0}(u+3)^3(8e)^{-u}.
\]
The final series is finite and depends on no parameter, proving the claim.
\end{proof}

\begin{lemma}[Higher-excess bound]\label{lem:higher-excess-bound}
For every $0\le C_0<\infty$ there is a constant
$C_{\mathrm{high}}=C_{\mathrm{high}}(C_0)$ such that, for all sufficiently
large $n$ in the setting of \cref{lem:truncation-parameters} and every
$t\ge n/2$,
\begin{equation}\label{eq:higher-excess-bound}
\sum_{\substack{r\le R_{\mathrm{tr}}\\g\ge2}}
\mathcal E_{r,g}(t)
\le
C_{\mathrm{high}}
\frac{\theta R_{\mathrm{tr}}^7}{n^2}
=
o\!\left(\frac{\theta}{n}\right).
\end{equation}
The little-$o$ term is uniform over the matrix class and over
$t\ge n/2$.
\end{lemma}

\begin{proof}
A connected bipartite incidence graph has at least one vertex in each
bipartition class.  Hence $v\ge2$, and therefore
$r-g-1=v-1\ge1$.  Since $\theta\le1$, the factor
$\theta^{r-g-1}$ in \eqref{eq:fixed-rg-bound} is at most $\theta$.

Put
\[
x=\frac{8\Lambda_0R_{\mathrm{tr}}^3}{n}.
\]
By \eqref{eq:Rtr-growth}, $x\le1/2$ for all sufficiently large $n$.
There are at most $R_{\mathrm{tr}}$ possible values of $r$, and extending
the excess sum only enlarges the bound.  Thus
\begin{align*}
\sum_{\substack{r\le R_{\mathrm{tr}}\\g\ge2}}
\mathcal E_{r,g}(t)
&\le
2\Lambda_0\theta R_{\mathrm{tr}}
\sum_{g\ge2}\frac{x^g}{(2g)!}\\
&\le
2\Lambda_0\theta R_{\mathrm{tr}}
\sum_{g\ge2}x^g\\
&\le
4\Lambda_0\theta R_{\mathrm{tr}}x^2\\
&=
256\Lambda_0^3
\frac{\theta R_{\mathrm{tr}}^7}{n^2}.
\end{align*}
The final equality in \eqref{eq:higher-excess-bound} follows from
$R_{\mathrm{tr}}^7/n\to0$.
\end{proof}

\begin{lemma}[Cluster truncation]\label{lem:cluster-truncation}
In the setting of \cref{lem:truncation-parameters}, for every $t\ge n/2$,
\begin{equation}\label{eq:cluster-truncation}
\sum_{r>R_{\mathrm{tr}}}\frac{t^{-r}}{r!}
 \sum_{e_1,\dots,e_r}\abs{\phi^T(e_1,\dots,e_r)}
 \prod_{a=1}^r\abs{E_{e_a}}
\le 8s\,n^{-4}.
\end{equation}
\end{lemma}

\begin{proof}
For the activities $w_{ij}=E_{ij}/t$, \eqref{eq:WDelta} gives
$W\le2s$ and $e\Delta\le q$.  Hence
\cref{lem:cluster-bound,lem:truncation-parameters} gives
\[
\text{left side of \eqref{eq:cluster-truncation}}
\le4W(e\Delta)^{R_{\mathrm{tr}}}
\le8s q^{R_{\mathrm{tr}}}
\le8s n^{-4}.
\]
\end{proof}

\begin{lemma}[Cycle-series truncation]\label{lem:cycle-truncation}
For every $0\le C_0<\infty$ there is a constant
$C_{\mathrm{cyc}}=C_{\mathrm{cyc}}(C_0)$ such that, in the setting of
\cref{lem:truncation-parameters} and for every $t\ge n/2$,
\begin{equation}\label{eq:cycle-truncation}
\sum_{2\ell>R_{\mathrm{tr}}}
\frac1{2\ell t^{2\ell}}
\tr\bigl((E^{\mathsf T}E)^\ell\bigr)
\le C_{\mathrm{cyc}}\frac{s}{n^2}.
\end{equation}
\end{lemma}

\begin{proof}
Since
\[
\tr(E^{\mathsf T}E)\le\Lambda_0ns,
\qquad
\norm{E}_{\op}\le s,
\]
one has
\[
\tr\bigl((E^{\mathsf T}E)^\ell\bigr)
\le \Lambda_0ns^{2\ell-1}.
\]
Let $\ell_0=\lfloor R_{\mathrm{tr}}/2\rfloor+1$.  Then
$2\ell_0>R_{\mathrm{tr}}$ and
$2\ell_0-2\ge R_{\mathrm{tr}}-1$.  Since
$\theta=s/t\le1/2$ for all sufficiently large $n$,
\begin{align}
\sum_{2\ell>R_{\mathrm{tr}}}
\frac1{2\ell t^{2\ell}}
\tr\bigl((E^{\mathsf T}E)^\ell\bigr)
&\le
\frac{\Lambda_0ns}{2t^2}
\sum_{\ell\ge\ell_0}\theta^{2\ell-2}\notag\\
&=
\frac{\Lambda_0ns}{2t^2}
\frac{\theta^{2\ell_0-2}}{1-\theta^2}\notag\\
&\le
\frac{8\Lambda_0}{3}\frac{s}{n}
\theta^{R_{\mathrm{tr}}-1}.
\label{eq:cycle-geometric}
\end{align}
If $s\le\sqrt n$, then $R_{\mathrm{tr}}\ge3$ and
\[
\frac{s}{n}\theta^{R_{\mathrm{tr}}-1}
\le
\frac{s}{n}\left(\frac{2s}{n}\right)^2
\le4\frac{s}{n^2}.
\]
If $s>\sqrt n$, then $\theta\le q$ and
\[
\frac{s}{n}\theta^{R_{\mathrm{tr}}-1}
\le
\frac{s}{n}q^{R_{\mathrm{tr}}-1}
=
\frac{q^{R_{\mathrm{tr}}}}{4e}
\le\frac{n^{-4}}{4e}
\le\frac{s}{n^2}.
\]
Combining the two cases proves \eqref{eq:cycle-truncation}.
\end{proof}

\begin{remark}[Error ledger]\label{rem:error-ledger}
For later reference, the fixed-$t$ logarithmic error is assembled as follows:
\[
\begin{array}{c|c}
\text{source} & \text{uniform bound}\\ \hline
\text{incidence excess }1
  & O_{C_0}(\theta/n)=O_{C_0}(s/n^2)\\
\text{incidence excess at least }2
  & O_{C_0}(\theta R_{\mathrm{tr}}^7/n^2)
    =o_{C_0}(s/n^2)\\
\text{cluster degrees above }R_{\mathrm{tr}}
  & O(s n^{-4})\\
\text{cycle degrees above }R_{\mathrm{tr}}
  & O_{C_0}(s/n^2).
\end{array}
\]
The lower Gamma tail contributes $O_{C_0}(se^{-c n})$, and the displacement
of the determinant potential from $T$ to $n$ contributes
$O_{C_0}(s/n^2)$; these are treated in
\cref{lem:lower-gamma-tail,lem:gamma-average-D}.
\end{remark}

For $t\ge n/2$, put
\begin{equation}\label{eq:D-def}
D(t)=-\frac12\log\det\!\left(I-\frac{E^{\mathsf T}E}{t^2}\right).
\end{equation}

\begin{proposition}[Fixed-$t$ determinant approximation]\label{prop:fixed-t}
For every $0\le C_0<\infty$ there is a constant
$C_{\mathrm{fix}}=C_{\mathrm{fix}}(C_0)$ such that, in the setting of
\cref{thm:small-line-permanent}, the exact difference
\begin{equation}\label{eq:fixed-t-error-def}
\varepsilon_E(t)
:=\log Z_E(t^{-1})
+\frac12\log\det\!\left(I-\frac{E^{\mathsf T}E}{t^2}\right)
\end{equation}
satisfies
\begin{equation}\label{eq:fixed-t-error}
\abs{\varepsilon_E(t)}
\le C_{\mathrm{fix}}\frac{s}{n^2}
\end{equation}
uniformly for every $t\ge n/2$.  The sufficiently-large-$n$ threshold
depends only on $C_0$ and the prescribed sequence $\eta_\bullet$, not on
$E$ or $t$.
\end{proposition}

\begin{proof}
Apply \cref{lem:cluster-bound} to the activities $E_{ij}/t$.  Let
$R_{\mathrm{tr}}$ be defined by \eqref{eq:Rtr-def}.  Absolute convergence
and \eqref{eq:cluster-truncation} give
\[
\log Z_E(t^{-1})
=\sum_{r=1}^{R_{\mathrm{tr}}}c_rt^{-r}+\mathcal R_{\mathrm{cl}}(t),
\qquad
c_r=[z^r]\log Z_E(z),
\]
where
\[
\abs{\mathcal R_{\mathrm{cl}}(t)}\le8s n^{-4}.
\]
For each $r\le R_{\mathrm{tr}}$,
\cref{lem:connected-partition} is a finite exact identity.  Let
$\mathcal P_{\le R_{\mathrm{tr}}}(t)$ denote the signed sum of all its
positive-incidence-excess terms through degree $R_{\mathrm{tr}}$, and let
\[
\mathcal R_{\mathrm{cyc}}(t)
=\sum_{2\ell>R_{\mathrm{tr}}}
 \frac{1}{2\ell t^{2\ell}}
 \tr\bigl((E^{\mathsf T}E)^\ell\bigr).
\]
The zero-incidence-excess terms through degree $R_{\mathrm{tr}}$ are the
corresponding truncated cycle series.  Since the complete cycle series is
$D(t)$, the exact signed decomposition is
\begin{equation}\label{eq:fixed-t-signed-decomposition}
\varepsilon_E(t)
=\mathcal P_{\le R_{\mathrm{tr}}}(t)
 +\mathcal R_{\mathrm{cl}}(t)
 -\mathcal R_{\mathrm{cyc}}(t).
\end{equation}
Here $D(t)$ contains the complete cycle series, while $\log Z_E(t^{-1})$
contains only its truncated part plus the cluster tail; this accounts for the
minus sign.

The absolute value of $\mathcal P_{\le R_{\mathrm{tr}}}(t)$ is bounded by
\eqref{eq:excess-one-bound} and \eqref{eq:higher-excess-bound}.  Since
$\theta/n=s/(tn)\le2s/n^2$, \eqref{eq:excess-one-bound} contributes at most
$2C_{\mathrm{one}}s/n^2$.  After increasing the uniform threshold so that
$R_{\mathrm{tr}}^7/n\le1$, \eqref{eq:higher-excess-bound} contributes at most
$2C_{\mathrm{high}}s/n^2$.  The cluster tail is at most
$8s n^{-4}\le8s/n^2$, and \eqref{eq:cycle-truncation} bounds the cycle tail
by $C_{\mathrm{cyc}}s/n^2$.  Taking absolute values in
\eqref{eq:fixed-t-signed-decomposition} proves \eqref{eq:fixed-t-error} with,
for example,
\[
C_{\mathrm{fix}}
=2C_{\mathrm{one}}+2C_{\mathrm{high}}+8+C_{\mathrm{cyc}}.
\]
This constant depends only on $C_0$ and is independent of
$\eta_\bullet$, $E$, and $t$.
\end{proof}

\subsection{Averaging over the Gamma variable}

The following estimates are immediate from the eigenvalue expansion.

\begin{lemma}[Bounds for the determinant potential]\label{lem:D-bounds}
There is a constant $C_D=C_D(C_0)$ such that, uniformly for $t\ge n/2$,
\begin{equation}\label{eq:D-bounds}
0\le D(t)\le C_D\frac{s}{n},
\qquad
\abs{D'(t)}\le C_D\frac{s}{n^2},
\qquad
\abs{D''(t)}\le C_D\frac{s}{n^3}.
\end{equation}
\end{lemma}

\begin{proof}
Let $\lambda_1,\dots,\lambda_n$ be the eigenvalues of
$E^{\mathsf T}E$.  Since $\lambda_i\le s^2$ and
\[
\sum_i\lambda_i=\norm{E}_{\F}^2
\le \Lambda_0\sum_{i,j}\abs{E_{ij}}
\le \Lambda_0ns,
\]
we have, uniformly for $t\ge n/2$,
\[
\max_i\frac{\lambda_i}{t^2}\le\frac{s^2}{t^2}=o(1).
\]
After increasing the uniform threshold, every denominator below is at least
$1/2$.  Direct differentiation gives
\[
D'(t)=-\sum_i\frac{\lambda_i}{t^3(1-\lambda_i/t^2)},
\]
\[
D''(t)=\sum_i\left(
\frac{3\lambda_i}{t^4(1-\lambda_i/t^2)}
+\frac{2\lambda_i^2}{t^6(1-\lambda_i/t^2)^2}\right).
\]
Consequently,
\[
D(t)=O_{C_0}(ns/t^2),
\qquad
\abs{D'(t)}=O_{C_0}(ns/t^3).
\]
Moreover,
$\sum_i\lambda_i^2\le s^2\sum_i\lambda_i\le\Lambda_0ns^3$, and hence
\[
D''(t)=O_{C_0}(ns/t^4+ns^3/t^6)=O_{C_0}(s/n^3),
\]
because $s/n=o(1)$.  This proves \eqref{eq:D-bounds}.
\end{proof}

\begin{lemma}[The lower Gamma tail is linearly negligible]\label{lem:lower-gamma-tail}
There are constants $c_{\Gamma}=c_{\Gamma}(C_0)>0$ and $C_{\Gamma}=C_{\Gamma}(C_0)$ such that
\begin{equation}\label{eq:lower-gamma-tail}
\E\bigl[\abs{Z_E(T^{-1})-1};\,T<n/2\bigr]
\le C_{\Gamma} s e^{-c_{\Gamma}n}
\end{equation}
for all sufficiently large $n$, uniformly over the matrix class.
\end{lemma}

\begin{proof}
For matchings of size $r$, first choose their $r$ row vertices.  If column
distinctness is ignored, the sum of the absolute products is at most $s^r$.
Hence
\[
\abs{Z_E(t^{-1})-1}\le(1+s/t)^n-1.
\]
Consequently,
\[
\E\bigl[\abs{Z_E(T^{-1})-1};\,T<n/2\bigr]
\le\frac1{n!}\int_0^{n/2}e^{-t}\bigl[(t+s)^n-t^n\bigr]dt.
\]
The mean-value theorem gives
$(t+s)^n-t^n\le ns(t+s)^{n-1}$.  Furthermore,
\begin{equation}\label{eq:lower-tail-monotonicity}
\frac{d}{dt}\log\!\left(e^{-t}(t+s)^{n-1}\right)
=\frac{n-1}{t+s}-1.
\end{equation}
Since $s/n\le\eta_n\to0$, one has $s\le n/2-1$ for all sufficiently
large $n$, uniformly over the class.  For $0\le t\le n/2$ this gives
$t+s\le n-1$, so the right side of
\eqref{eq:lower-tail-monotonicity} is nonnegative.  Thus, using
$n!\ge(n/e)^n$,
\[
\E\bigl[\abs{Z_E(T^{-1})-1};\,T<n/2\bigr]
\le
\frac{n^2s}{2n!}e^{-n/2}(n/2+s)^{n-1}.
\]
Fix a number $\rho$ with $e^{1/2}/2<\rho<1$.  Uniformly over the class,
$e^{1/2}(1/2+s/n)\le\rho$ for all sufficiently large $n$.  The last
display is then at most $C_{\Gamma}ns\rho^n$, and hence at most
$C_{\Gamma}se^{-c_{\Gamma}n}$ after decreasing $c_{\Gamma}$ if necessary.
\end{proof}

\begin{lemma}[Gamma averaging about $n$]\label{lem:gamma-average-D}
One has
\begin{equation}\label{eq:gamma-average-D}
\E\bigl[e^{D(T)}-1;\,T\ge n/2\bigr]
=e^{D(n)}-1+O_{C_0}(s/n^2).
\end{equation}
\end{lemma}

\begin{proof}
Although $\E T=n+1$, the determinant target in the permanent theorem is $D(n)$.  The unit first-order displacement contributes $D'(n)\E(T-n)=D'(n)=O_{C_0}(s/n^2)$, which is already within the claimed remainder.  Let $\delta=T-n$ and $\mathcal A=\{T\ge n/2\}$.  The Gamma distribution
has
\[
\E\delta=1,
\qquad
\E\delta^2=n+2,
\qquad
\E\delta^4=3n^2+26n+24.
\]
Also, exponential Markov inequality with parameter $1$ gives
\[
\mathbb P(\mathcal A^c)
\le e^{n/2}\E e^{-T}
=\frac{e^{n/2}}{2^{n+1}}
\le e^{-c n}
\]
for an absolute $c>0$.  For $j=1,2$, Cauchy--Schwarz gives the explicit
truncation estimate
\begin{equation}\label{eq:gamma-truncated-moments}
\E\bigl[\abs\delta^j;\mathcal A^c\bigr]
\le
\bigl(\E\abs\delta^{2j}\bigr)^{1/2}
\mathbb P(\mathcal A^c)^{1/2}
\le C_j n^{j/2}e^{-c n/2}.
\end{equation}

Set $X=D(T)-D(n)$ on $\mathcal A$.  Since $D$ is decreasing and
\cref{lem:D-bounds} applies on $[n/2,\infty)$,
\begin{equation}\label{eq:X-uniform}
\abs X\le D(n/2)=O_{C_0}(s/n)=o(1).
\end{equation}
Taylor's theorem and \cref{lem:D-bounds} yield
\[
X=D'(n)\delta+O_{C_0}\!\left(\frac{s}{n^3}\delta^2\right)
\quad\text{on }\mathcal A.
\]
Using \eqref{eq:gamma-truncated-moments}, every omitted lower-tail moment is
multiplied by $D'(n)=O_{C_0}(s/n^2)$ or by $O_{C_0}(s/n^3)$.  Hence
\begin{align*}
\E[X;\mathcal A]
&=D'(n)\left(1+O(\sqrt n\,e^{-cn/2})\right)
 +O_{C_0}\!\left(\frac{s}{n^3}
      (n+2+O(ne^{-cn/2}))\right)\\
&=O_{C_0}(s/n^2).
\end{align*}
Similarly,
\[
\E[X^2;\mathcal A]
=O_{C_0}\!\left(
\frac{s^2}{n^4}\E\delta^2
+\frac{s^2}{n^6}\E\delta^4\right)
=O_{C_0}(s^2/n^3)
=o_{C_0}(s/n^2).
\]
By \eqref{eq:X-uniform}, $e^X=1+X+O(X^2)$ uniformly on $\mathcal A$, so
\[
\E[e^X;\mathcal A]
=\mathbb P(\mathcal A)+O_{C_0}(s/n^2).
\]
Therefore
\begin{align*}
\E[e^{D(T)}-1;\mathcal A]
&=e^{D(n)}\E[e^X;\mathcal A]-\mathbb P(\mathcal A)\\
&=\mathbb P(\mathcal A)(e^{D(n)}-1)+O_{C_0}(s/n^2).
\end{align*}
Finally, $e^{D(n)}-1=O_{C_0}(s/n)$ and
$\mathbb P(\mathcal A^c)\le e^{-cn}$.  Replacing
$\mathbb P(\mathcal A)$ by $1$ therefore costs only
$O_{C_0}(se^{-cn}/n)$, which is absorbed by $O_{C_0}(s/n^2)$.
\end{proof}

\begin{proof}[Proof of \cref{thm:small-line-permanent}]
By \cref{prop:fixed-t}, on $\mathcal A=\{T\ge n/2\}$ the exact error $\varepsilon_E(T)$ from \eqref{eq:fixed-t-error-def} satisfies
$\abs{\varepsilon_E(T)}\le C_{\mathrm{fix}}s/n^2$ and
\[
Z_E(T^{-1})=\exp\bigl(D(T)+\varepsilon_E(T)\bigr).
\]
Since $D(T)=O_{C_0}(s/n)$ on $\mathcal A$,
\[
\E[Z_E(T^{-1})-1;\mathcal A]
=\E[e^{D(T)}-1;\mathcal A]+O_{C_0}(s/n^2).
\]
Using the centred decomposition
\[
\E Z_E(T^{-1})
=1+\E[Z_E(T^{-1})-1;\mathcal A]
  +\E[Z_E(T^{-1})-1;\mathcal A^c]
\]
and then applying
\cref{lem:gamma-representation,lem:lower-gamma-tail,lem:gamma-average-D},
we obtain
\[
\frac{\per(\Jall+E)}{n!}
=1+e^{D(n)}-1+O_{C_0}(s/n^2)
=e^{D(n)}+O_{C_0}(s/n^2).
\]
Because $D(n)\ge0$, one has $e^{D(n)}\ge1$; dividing by $e^{D(n)}$ therefore does not enlarge the additive error.  Since $s/n^2=o(1)$, the preceding display may be written as
\[
\frac{\per(\Jall+E)}{n!}=e^{D(n)}(1+\zeta_n),
\qquad
\abs{\zeta_n}\le C_{\mathrm{perm}} s/n^2.
\]
For all sufficiently large $n$, $\abs{\zeta_n}\le1/2$, so
$\log(1+\zeta_n)=O_{C_0}(s/n^2)$.  This proves
\eqref{eq:small-line-permanent}.
\end{proof}

\section{Latin rectangles and residual permanents}\label{sec:graph-setup}

Let $R$ be an $m\times n$ Latin rectangle.  Its \emph{used column--symbol matrix} is the zero--one matrix
\begin{equation}\label{eq:M-def}
M=M(R),
\qquad
M_{cs}=1
\quad\Longleftrightarrow\quad
\text{symbol $s$ occurs in column $c$ of $R$}.
\end{equation}
Every row sum of $M$ equals $m$ because each array column contains $m$ distinct symbols, and every column sum of $M$ equals $m$ because each symbol occurs once in each of the $m$ Latin rows.  Put
\begin{equation}\label{eq:dC-def}
d=n-m,
\qquad
\Proj=\frac1n\1\1^{\mathsf T},
\qquad
C=M-m\Proj.
\end{equation}
Then
\begin{equation}\label{eq:C-zero-sums}
C\1=C^{\mathsf T}\1=0.
\end{equation}
The adjacency matrix of the residual graph of allowed positions is
\begin{equation}\label{eq:allowed-matrix}
A=\Jall-M=d\Proj-C.
\end{equation}
It is a $d$-regular zero--one matrix.

\begin{lemma}[Extension recurrence]\label{lem:extension-recurrence}
The number of rows that can be appended to $R$ is
\begin{equation}\label{eq:extension-permanent}
X(R)=\per(A).
\end{equation}
Consequently, if $R_m$ is uniformly random among the $m\times n$ Latin rectangles, then
\begin{equation}\label{eq:exact-recurrence}
\frac{L_{m+1,n}}{L_{m,n}}=\E X(R_m).
\end{equation}
\end{lemma}

\begin{proof}
An additional row must choose exactly one allowed symbol in each column and use every symbol exactly once.  Such a choice is precisely a perfect matching of the residual column--symbol graph.  Summing the extension counts over all $m\times n$ rectangles proves \eqref{eq:exact-recurrence}.
\end{proof}

Normalize the residual matrix by
\begin{equation}\label{eq:P-def}
P=\frac{A}{d}=\Proj-\frac{C}{d}.
\end{equation}
Then $P$ is doubly stochastic and
\begin{equation}\label{eq:E-residual}
nP=\Jall+E,
\qquad
E=-\frac{n}{d}C.
\end{equation}

\begin{lemma}[Line bounds for the residual deviation]\label{lem:residual-line-bounds}
Assume $m\le n/2$.  The matrix $E$ in \eqref{eq:E-residual} satisfies
\begin{equation}\label{eq:residual-line-bounds}
E\1=E^{\mathsf T}\1=0,
\qquad
\Jall+E\ge0,
\qquad
\max_{c,s}\abs{E_{cs}}\le1,
\qquad
s(E)=2m.
\end{equation}
\end{lemma}

\begin{proof}
On a used cell, $A_{cs}=0$ and hence $E_{cs}=-1$.  On an allowed cell,
\[
E_{cs}=\frac nd-1=\frac md\le1.
\]
Every row and column contains $m$ used and $d$ allowed cells, so its absolute sum is
\[
m+d\frac md=2m.
\]
The zero-sum and nonnegativity assertions follow from \eqref{eq:E-residual}.
\end{proof}

\begin{proposition}[Pointwise residual permanent]\label{prop:pointwise-permanent}
Let $K:\mathbb N\to\mathbb Z_{\ge0}$ satisfy $K(n)=o(n)$.  As $n\to\infty$, uniformly over all $m\times n$ Latin rectangles with $0\le m\le K(n)$,
\begin{equation}\label{eq:pointwise-permanent}
\log X(R)
=\log(n!)+n\log\frac dn
 -\frac12\log\det\!\left(I-\frac{C^{\mathsf T}C}{d^2}\right)
 +O\!\left(\frac{m}{n^2}\right).
\end{equation}
\end{proposition}

\begin{proof}
For all sufficiently large $n$, $m\le n/2$.  Apply \cref{thm:small-line-permanent} to \eqref{eq:E-residual}, using \cref{lem:residual-line-bounds} and the choice $\eta_n=2K(n)/n$.  Since
\[
\frac{E^{\mathsf T}E}{n^2}=\frac{C^{\mathsf T}C}{d^2}
\]
and
\[
X(R)=\per(A)=\left(\frac dn\right)^n\per(nP),
\]
formula \eqref{eq:pointwise-permanent} follows.
\end{proof}

\section{A fixed-anchor incomplete-column-path switching}\label{sec:switching}

The local operation used below is the standard incomplete two-column
cycle/path switch, written in directed-path form; see Wanless
\cite{WanlessCycleSwitches} and the Latin-rectangle formulation of Allsop and
Wanless \cite[Section~2]{AllsopWanless}.  The points specific to the present
argument are the fixed-anchor sampling scheme, the resulting state-independent number
of decorated moves, the exact exchangeable-pair regression, and the blocker
term retained in the variance identity.

Throughout this section assume $1\le m<n$.  For a column $a$, let
\[
S_a=S_a(R)=\{s\in[n]:M_{as}=1\}
\]
be the set of symbols used in that column, so $\abs{S_a}=m$.

\subsection{The path involution}

Fix distinct columns $a$ and $h$.  For every row $r$, draw a directed edge
\[
R_{r,a}\longrightarrow R_{r,h}
\]
on the symbol set $[n]$.  Because no symbol is repeated in either column, every symbol has indegree at most one and outdegree at most one.  Hence every nonisolated component is a directed cycle or a directed path.  Every path starts in $S_a\setminus S_h$ and ends in $S_h\setminus S_a$.

\begin{lemma}[Path switch]\label{lem:path-switch}
Let
\[
Q:\quad x_0\longrightarrow x_1\longrightarrow\cdots\longrightarrow x_t
\]
be a directed path in the two-column digraph for $(a,h)$.  For
$i=1,\dots,t$, let $r_i$ be the unique row supplying the edge
$x_{i-1}\to x_i$.  In every row $r_i$, interchange the entries in columns
$a$ and $h$.  The resulting array $R^Q$ is a Latin rectangle and
\begin{equation}\label{eq:path-set-change}
S_a(R^Q)=S_a(R)-x_0+x_t,
\qquad
S_h(R^Q)=S_h(R)-x_t+x_0.
\end{equation}
In the $(a,h)$-digraph of $R^Q$, the same rows supply the reversed path
\[
Q^{\leftarrow}:\quad
x_t\longrightarrow x_{t-1}\longrightarrow\cdots\longrightarrow x_0,
\]
and $Q^{\leftarrow}$ is a whole path component.
\end{lemma}

\begin{proof}
The local change is represented schematically by
\[
\begin{array}{c|cc|cc}
&\multicolumn{2}{c|}{R}&\multicolumn{2}{c}{R^Q}\\
& a&h&a&h\\ \hline
r_1&x_0&x_1&x_1&x_0\\
r_2&x_1&x_2&x_2&x_1\\
\vdots&\vdots&\vdots&\vdots&\vdots\\
r_t&x_{t-1}&x_t&x_t&x_{t-1}.
\end{array}
\]
Each affected row undergoes a transposition and remains a permutation.  In
column $a$, the internal symbols $x_1,\dots,x_{t-1}$ are merely moved
between affected rows, while $x_0$ is removed and $x_t$ is inserted.
Column $h$ behaves oppositely.  Since a path begins in
$S_a\setminus S_h$ and ends in $S_h\setminus S_a$, the two inserted
symbols were absent from the columns receiving them.  Thus no column
repetition is created, proving that $R^Q$ is Latin and giving
\eqref{eq:path-set-change}.

In row $r_i$, the directed edge in the new $(a,h)$-digraph is
$x_i\to x_{i-1}$.  All rows outside $Q$ are unchanged.  The original
endpoint $x_t$ had no outgoing edge and the original endpoint $x_0$ had no
incoming edge; after reversal they have, respectively, no incoming and no
outgoing edge.  Each internal vertex still has indegree and outdegree one.
Hence no unchanged edge joins the displayed reversed path, so
$Q^{\leftarrow}$ is a complete path component of the new digraph.
\end{proof}

Fix an anchor column $a$.  For every partner $h\ne a$, every path component
of the $(a,h)$-digraph is an available anchored switch.

\begin{lemma}[State-independent number of anchored paths]\label{lem:path-count}
Every $m\times n$ Latin rectangle has exactly
\begin{equation}\label{eq:path-count}
m(n-m)=md
\end{equation}
anchored path switches at a fixed anchor column $a$.
\end{lemma}

\begin{proof}
For a fixed partner $h$, the number of paths is
$\abs{S_a\setminus S_h}$.  Therefore
\[
\sum_{h\ne a}\abs{S_a\setminus S_h}
=m(n-1)-\sum_{h\ne a}\abs{S_a\cap S_h}.
\]
Every symbol in $S_a$ occurs in exactly $m-1$ other columns, so the final
sum is $m(m-1)$.  The result is $m(n-m)$.
\end{proof}

Let $\mathcal L_{m,n}$ be the set of $m\times n$ Latin rectangles and
define the decorated state space
\begin{equation}\label{eq:decorated-space}
\mathfrak D_a
=\left\{(R,h,Q):
R\in\mathcal L_{m,n},\ h\ne a,\
Q\text{ is a path component of the $(a,h)$-digraph of }R
\right\}.
\end{equation}
If $R^Q$ is obtained by switching $Q$, define
\begin{equation}\label{eq:decorated-involution}
\Phi(R,h,Q)=(R^Q,h,Q^{\leftarrow}).
\end{equation}

\begin{lemma}[Decorated involution and exchangeability]\label{lem:exchangeable}
The map $\Phi$ is an involution of $\mathfrak D_a$.  Every rectangle has
exactly $md$ decorations, and therefore the uniform measure on
$\mathfrak D_a$ is $\Phi$-invariant.  If $(R,h,Q)$ is uniform on
$\mathfrak D_a$ and $R'=R^Q$, then $(R,R')$ is exchangeable.
\end{lemma}

\begin{proof}
By \cref{lem:path-switch}, $Q^{\leftarrow}$ is a path component for the
same ordered column pair $(a,h)$ in $R^Q$; in particular, the fixed anchor
remains $a$.  Switching $Q^{\leftarrow}$ reverses the same row
transpositions and recovers both $R$ and the original orientation of $Q$.
Thus $\Phi^2$ is the identity.  By \cref{lem:path-count}, every
$R\in\mathcal L_{m,n}$ has exactly $md$ decorations, so choosing $R$
uniformly and then one of its decorations uniformly is the uniform law on
$\mathfrak D_a$.  An involution preserves this uniform law.  Under
$\Phi$, the ordered pair of rectangle coordinates changes from
$(R,R^Q)$ to $(R^Q,R)$, and hence the projected pair $(R,R')$ is
exchangeable.
\end{proof}

\subsection{Codegree regression and variance}

Fix another column $b\ne a$ and set
\begin{equation}\label{eq:Z-def}
Z=\abs{S_a\cap S_b}.
\end{equation}
Let $Z'$ be the codegree after the random anchored switch and put $\Delta=Z'-Z$.

\begin{proposition}[Exact codegree regression]\label{prop:codegree-regression}
Conditionally on $R$,
\begin{equation}\label{eq:codegree-regression}
\E[\Delta\mid R]
=-\frac{n-1}{md}
 \left(Z-\frac{m(m-1)}{n-1}\right).
\end{equation}
Consequently,
\begin{equation}\label{eq:codegree-mean}
\E Z=\frac{m(m-1)}{n-1}.
\end{equation}
\end{proposition}

\begin{proof}
For a chosen path, let $x$ be its initial symbol and $y$ its terminal symbol, and let $h$ be the partner column.  If $h\ne b$, then
\[
\Delta=\1_{\{y\in S_b\}}-\1_{\{x\in S_b\}}.
\]
If $h=b$, the two sets $S_a$ and $S_b$ exchange the same endpoint symbols and $Z$ is unchanged.

Every $x\in S_a\cap S_b$ is absent from exactly $d$ columns, none of which is $b$.  For each such partner, $x$ starts a unique path.  Hence the total number of negative endpoint incidences is $Zd$.

Every $y\in S_b\setminus S_a$ occurs in exactly $m$ columns.  For each partner column containing $y$, the symbol $y$ terminates a unique path.  The partner $h=b$ causes no change in $Z$, leaving $m-1$ positive incidences.  Hence the total number of positive endpoint incidences is $(m-Z)(m-1)$.

Dividing their difference by $md$ proves \eqref{eq:codegree-regression}.  Exchangeability gives $\E\Delta=0$, and hence \eqref{eq:codegree-mean}.
\end{proof}

For a path $Q$, write $\partial Q$ for its two endpoint symbols.  Define
the anchored blocker count
\begin{equation}\label{eq:blocker-def}
B_{ab}(R)
=\#\bigl\{(h,Q):(R,h,Q)\in\mathfrak D_a,\
                     \ \partial Q\subseteq S_b\bigr\}.
\end{equation}
The partner $h=b$ is included in the definition but contributes nothing,
because a path in the $(a,b)$-digraph starts in $S_a\setminus S_b$.

\begin{proposition}[Exact variance defect]\label{prop:variance-defect}
One has
\begin{equation}\label{eq:variance-defect}
\operatorname{Var}Z
=\frac{m(m-1)d}{(n-1)^2}
 -\frac{\E B_{ab}}{n-1}.
\end{equation}
In particular,
\begin{equation}\label{eq:variance-bound}
\operatorname{Var}Z
\le \frac{m(m-1)d}{(n-1)^2}.
\end{equation}
\end{proposition}

\begin{proof}
The negative and positive incidence counts in the preceding proof are $Zd$ and $(m-Z)(m-1)$.  A path whose two endpoints both lie in $S_b$ is counted once in each incidence count but has $\Delta=0$.  Therefore
\begin{equation}\label{eq:conditional-delta-square}
\E[\Delta^2\mid R]
=\frac{Zd+(m-Z)(m-1)-2B_{ab}(R)}{md}.
\end{equation}
Put
\[
\mu=\frac{m(m-1)}{n-1},
\qquad
\lambda=\frac{n-1}{md}.
\]
By \eqref{eq:codegree-regression}, $\E[\Delta\mid R]=-\lambda(Z-\mu)$.  Exchangeability gives
\[
0=\E\bigl[(Z'-\mu)^2-(Z-\mu)^2\bigr]
=-2\lambda\operatorname{Var}Z+\E\Delta^2.
\]
Taking expectations in \eqref{eq:conditional-delta-square} and using $\E Z=\mu$ yields
\[
\E\Delta^2
=\frac{2m(m-1)d/(n-1)-2\E B_{ab}}{md}.
\]
Substitution gives \eqref{eq:variance-defect}.
\end{proof}

\section{The fourth Schatten moment of the centred incidence matrix}\label{sec:fourth-moment}

Continue to write $M$ for the used incidence matrix and $C=M-m\Proj$.  For distinct columns $a,b$, put
\[
c_{ab}=\abs{S_a\cap S_b}=(MM^{\mathsf T})_{ab},
\qquad
c_{aa}=m,
\qquad
\bar c=\frac{m(m-1)}{n-1}.
\]

\begin{lemma}[Exact codegree--spectral identity]\label{lem:T4-identity}
One has
\begin{equation}\label{eq:T4-identity}
\tr\bigl((C^{\mathsf T}C)^2\bigr)
=\sum_{a\ne b}(c_{ab}-\bar c)^2
 +\frac{m^2d^2}{n-1}.
\end{equation}
\end{lemma}

\begin{proof}
Since $CC^{\mathsf T}=MM^{\mathsf T}-m^2\Proj$ and
$\tr((C^{\mathsf T}C)^2)=\tr((CC^{\mathsf T})^2)$,
\begin{equation}\label{eq:T4-entry}
\tr\bigl((C^{\mathsf T}C)^2\bigr)
=\sum_{a,b}\left(c_{ab}-\frac{m^2}{n}\right)^2.
\end{equation}
For every fixed $a$,
\[
\sum_{b\ne a}c_{ab}=m(m-1)=(n-1)\bar c.
\]
Thus
\[
\sum_{a\ne b}\left(c_{ab}-\frac{m^2}{n}\right)^2
=\sum_{a\ne b}(c_{ab}-\bar c)^2
+n(n-1)\left(\bar c-\frac{m^2}{n}\right)^2.
\]
Now
\[
\bar c-\frac{m^2}{n}=-\frac{md}{n(n-1)},
\]
while the diagonal contribution in \eqref{eq:T4-entry} is
\[
n\left(m-\frac{m^2}{n}\right)^2=\frac{m^2d^2}{n}.
\]
The two deterministic terms sum to $m^2d^2/(n-1)$.
\end{proof}

\begin{corollary}[Expected fourth Schatten moment]\label{cor:T4-bound}
For a uniformly random $m\times n$ Latin rectangle with $1\le m<n$,
fix two distinct columns and let $B_{12}$ be the anchored blocker statistic from
\cref{prop:variance-defect}.  Then
\begin{equation}\label{eq:T4-exact-expectation}
\E\tr\bigl((C^{\mathsf T}C)^2\bigr)
=
\frac{n m(m-1)d+m^2d^2}{n-1}-n\E B_{12}.
\end{equation}
Consequently,
\begin{equation}\label{eq:T4-exact-bound}
\E\tr\bigl((C^{\mathsf T}C)^2\bigr)
\le
\frac{n m(m-1)d+m^2d^2}{n-1}.
\end{equation}
At $m=0$ or $m=n$, one has $C=0$, and the following bound remains valid.  In particular,
\begin{equation}\label{eq:T4-bigO}
\E\tr\bigl((C^{\mathsf T}C)^2\bigr)=O(nm^2)
\end{equation}
uniformly for $0\le m\le n$.
\end{corollary}

\begin{proof}
For $1\le m<n$, column exchangeability and \eqref{eq:codegree-mean} give
\[
\E\sum_{a\ne b}(c_{ab}-\bar c)^2
=n(n-1)\operatorname{Var}(c_{12}).
\]
Substitute the exact variance identity \eqref{eq:variance-defect} into
\eqref{eq:T4-identity} to obtain
\eqref{eq:T4-exact-expectation}; dropping the nonnegative blocker term
gives \eqref{eq:T4-exact-bound}.  The endpoint cases are immediate from
$C=0$.
\end{proof}

\section{A one-row extension theorem}\label{sec:one-step}

Let
\begin{equation}\label{eq:Q-def}
Q=\frac{C^{\mathsf T}C}{d^2}.
\end{equation}
It is positive semidefinite and annihilates the constant vector.

\begin{lemma}[First trace and spectral radius]\label{lem:Q-trace}
One has
\begin{equation}\label{eq:Q-trace}
\tr Q=\frac{m}{d}
\end{equation}
and
\begin{equation}\label{eq:Q-norm}
\norm{Q}_{\op}\le\left(\frac{m}{d}\right)^2.
\end{equation}
\end{lemma}

\begin{proof}
Since $M$ has $nm$ ones,
\[
\norm{C}_{\F}^2
=\norm{M-m\Proj}_{\F}^2
=nm-m^2=md,
\]
which proves \eqref{eq:Q-trace}.  Also
\[
C=(I-\Proj)M(I-\Proj),
\]
so orthogonal projection cannot increase the operator norm.  Therefore
\[
\norm{C}_{\op}\le\norm{M}_{\op}
\le\sqrt{\norm{M}_1\norm{M}_\infty}=m,
\]
giving \eqref{eq:Q-norm}.
\end{proof}

Define
\begin{equation}\label{eq:U-def}
U(R)
=-\frac12\log\det(I-Q)-\frac12\tr Q.
\end{equation}

\begin{lemma}[Determinant remainder]\label{lem:U-bound}
Let $K:\mathbb N\to\mathbb Z_{\ge0}$ satisfy $K(n)=o(n)$ and let
$0\le m\le K(n)$.  Uniformly over all $m\times n$ Latin rectangles,
\begin{equation}\label{eq:U-positive}
U(R)\ge0,
\end{equation}
\begin{equation}\label{eq:U-pointwise}
U(R)=O\!\left(\frac{m^3}{n^3}\right),
\end{equation}
and, for a uniform rectangle,
\begin{equation}\label{eq:EU}
\E U(R)=O\!\left(\frac{m^2}{n^3}\right).
\end{equation}
\end{lemma}

\begin{proof}
Let $q_1,\dots,q_n$ be the eigenvalues of $Q$.  By \eqref{eq:Q-norm}, $\max q_i=o(1)$ uniformly in the stated range.  Hence
\[
U(R)=\frac12\sum_{r\ge2}\frac{\tr(Q^r)}{r}\ge0
\]
and
\begin{equation}\label{eq:U-trQ2}
U(R)\le\frac{\tr(Q^2)}{4(1-\norm{Q}_{\op})}.
\end{equation}
Since
\[
\tr(Q^2)\le\norm{Q}_{\op}\tr Q
\le\left(\frac{m}{d}\right)^3,
\]
formula \eqref{eq:U-pointwise} follows.  Since
\[
\tr(Q^2)=d^{-4}\tr\bigl((C^{\mathsf T}C)^2\bigr),
\]
taking expectations in \eqref{eq:U-trQ2}, applying \cref{cor:T4-bound}, and using $d\sim n$ uniformly gives \eqref{eq:EU}.
\end{proof}

\begin{theorem}[One-row extension rigidity]\label{thm:one-step-L1}
Let $K:\mathbb N\to\mathbb Z_{\ge0}$ satisfy $K(n)=o(n)$, let $0\le m\le K(n)$, and let $R_m$ be uniformly random among the $m\times n$ Latin rectangles.  Put $d=n-m$ and
\begin{equation}\label{eq:Bmn-def}
B_{m,n}
=n!\left(\frac dn\right)^n
 \exp\!\left(\frac{m}{2d}\right).
\end{equation}
Then, as $n\to\infty$, uniformly in the stated range,
\begin{equation}\label{eq:L1-extension}
\E\left|\frac{X(R_m)}{B_{m,n}}-1\right|
=O\!\left(\frac{m}{n^2}\right).
\end{equation}
Consequently,
\begin{equation}\label{eq:one-step-ratio}
\frac{L_{m+1,n}}{L_{m,n}}
=B_{m,n}
\left[1+O\!\left(\frac{m}{n^2}\right)\right].
\end{equation}
\end{theorem}

\begin{proof}
For $m=0$, both assertions are exact.  Assume $m\ge1$.  By \cref{prop:pointwise-permanent,lem:Q-trace}, there is a remainder $\delta(R)$ satisfying the deterministic uniform bound
\[
\abs{\delta(R)}\le c_{\mathrm{ext}}\frac{m}{n^2}
\]
for an absolute constant $c_{\mathrm{ext}}$.  Here the permanent theorem is applied with $C_0=1$, so its error constant is the fixed number $C_{\mathrm{perm}}(1)$; only the sufficiently-large-$n$ threshold depends on the prescribed range function $K$.  Thus
\[
\log\frac{X(R)}{B_{m,n}}=U(R)+\delta(R).
\]
By \eqref{eq:U-pointwise}, the exponent tends uniformly to zero.  For
all sufficiently large $n$ it has absolute value at most one.  The inequality
$\abs{e^x-1}\le e^{\abs x}\abs x$, together with $U(R)\ge0$, therefore gives
\[
\left|e^{U(R)+\delta(R)}-1\right|
\le e\bigl(U(R)+\abs{\delta(R)}\bigr).
\]
Taking expectations and using \eqref{eq:EU} gives \eqref{eq:L1-extension}.  Equation \eqref{eq:one-step-ratio} follows from \eqref{eq:exact-recurrence}.
\end{proof}

\section{Telescoping the one-row estimate}\label{sec:telescoping}

\begin{lemma}[Exact products]\label{lem:exact-products}
For $0\le k<n$,
\begin{equation}\label{eq:product-factorial}
\prod_{m=0}^{k-1}
 n!\left(1-\frac mn\right)^n
=(n!)^k\left(\frac{(n)_k}{n^k}\right)^n,
\end{equation}
and
\begin{equation}\label{eq:harmonic-sum}
\sum_{m=0}^{k-1}\frac{m}{n-m}
=n(H_n-H_{n-k})-k.
\end{equation}
\end{lemma}

\begin{proof}
The first identity follows from
\[
\prod_{m=0}^{k-1}\left(1-\frac mn\right)=\frac{(n)_k}{n^k}.
\]
For the second, use $m/(n-m)=n/(n-m)-1$.
\end{proof}

\begin{proof}[Proof of \cref{thm:main}]
For $m<k$, write
\[
\frac{L_{m+1,n}}{L_{m,n}}=B_{m,n}(1+\varepsilon_{m,n}),
\]
where \cref{thm:one-step-L1} gives
\[
\abs{\varepsilon_{m,n}}
\le c_{\mathrm{ext}}\frac{m}{n^2}
\]
for the same absolute constant after enlargement if necessary.
The factor $1+\varepsilon_{m,n}$ is positive because it is the positive
extension ratio $L_{m+1,n}/L_{m,n}$ divided by $B_{m,n}>0$.  The displayed
bound is $o(1)$ uniformly for $m\le K(n)$, so
$\abs{\varepsilon_{m,n}}\le1/2$ for all sufficiently large $n$, and
$\abs{\log(1+\varepsilon_{m,n})}\le2\abs{\varepsilon_{m,n}}$.  Since
\[
\sum_{m=0}^{k-1}m=\frac{k(k-1)}2,
\]
we obtain explicitly
\[
\sum_{m=0}^{k-1}\log(1+\varepsilon_{m,n})
=O\!\left(\frac{k(k-1)}{2n^2}\right)
=O\!\left(\frac{k^2}{n^2}\right).
\]
By \cref{lem:exact-products}, the product of the $B_{m,n}$ is exactly $\widetilde A_{k,n}$.  This proves \eqref{eq:main-tilde}.

It remains to compare $\widetilde A_{k,n}$ and $A_{k,n}$.  For every integer $j\ge2$, monotonicity of $x\mapsto1/x$ gives
\[
0\le \int_{j-1}^{j}\frac{dx}{x}-\frac1j
\le \frac1{j-1}-\frac1j.
\]
Summing this inequality for $j=n-k+1,\dots,n$ yields
\begin{equation}\label{eq:harmonic-integral-bound}
0\le
\log\frac{n}{n-k}-(H_n-H_{n-k})
\le \frac1{n-k}-\frac1n
=\frac{k}{n(n-k)}.
\end{equation}
Consequently,
\[
\left|\log\frac{\widetilde A_{k,n}}{A_{k,n}}\right|
\le\frac{k}{2(n-k)}=O(k/n)
\]
uniformly for $k\le K(n)=o(n)$.  Combining this with \eqref{eq:main-tilde} proves \eqref{eq:main-GM}.
\end{proof}

\begin{proof}[Proof of \cref{cor:sublinear}]
By \eqref{eq:main-tilde},
\[
\widetilde x_{k,n}:=\log\frac{L_{k,n}}{\widetilde A_{k,n}}
=O(k^2/n^2)=o(1),
\]
and by \eqref{eq:main-GM},
\[
x_{k,n}:=\log\frac{L_{k,n}}{A_{k,n}}
=O(k/n)=o(1),
\]
uniformly on every sublinear range.  Since
$\abs{e^x-1}\le e^{\abs x}\abs x$, exponentiating the two displays gives
\eqref{eq:sublinear-relative-tilde} and \eqref{eq:sublinear-relative}.
\end{proof}

\section{Concluding remarks}

The proof separates the enumeration problem into a universal analytic component and a Latin-specific probabilistic component.  The permanent theorem resums the complete cycle sector of the matching expansion, while the path switching controls the first nonlinear singular-value statistic of the random residual graph.  Their error terms are summable over every sublinear number of rows.

The present argument does not by itself reach positive density.  In the permanent theorem, the matching-cluster expansion requires the maximum row or column $\ell_1$-norm $s(E)$ to be $o(n)$; for a Latin rectangle at height $m$, that line norm is exactly $2m$.  Within the present framework, treating positive-density or near-square regimes would therefore require a different analytic expansion or an additional probabilistic reduction.

\appendix
\section{Exact low-order and finite-height checks}\label{app:checks}

The identities in this appendix are not used in the proof of
\cref{thm:main}; they provide independent normalization checks for the two
load-bearing transitions in the argument.

\subsection{Cluster coefficients}
The first three coefficients were established abstractly in
\cref{lem:first-cluster-coefficients}.  In particular, they verify the
$1/r!$ normalization, both partition-lattice M\"obius factors, the sign of the
doubled-edge cycle at degree two, and the first positive-incidence-excess
term at degree three.

\subsection{Height two}
For $k=2$, the second row is a derangement of the first, so
$L_{2,n}=n!D_n$.  Since
\[
\frac{D_n}{n!}=e^{-1}+O((n+1)!^{-1}),
\]
and
\[
\log\widetilde A_{2,n}
=2\log(n!)+n\log\left(1-\frac1n\right)
 +\frac12\left(\frac1{n-1}\right),
\]
Taylor expansion gives
\[
\log\frac{L_{2,n}}{\widetilde A_{2,n}}
=-\frac1{6n^2}-\frac1{4n^3}-\frac3{10n^4}+O(n^{-5}),
\]
as stated in \eqref{eq:k2-check}.

\subsection{An exact height-three check}
After relabelling symbols, fix the first row to be the identity permutation.
The second row is then a derangement $\pi$.  For each such $\pi$, let
$A_\pi$ be the zero--one matrix
\[
(A_\pi)_{c,s}=\1_{\{s\notin\{c,\pi(c)\}\}}.
\]
The third row is exactly a perfect matching of $A_\pi$, and therefore
\begin{equation}\label{eq:k3-exact-check}
L_{3,n}=n!\sum_{\pi\in\mathfrak D_n}\per(A_\pi),
\end{equation}
where $\mathfrak D_n$ is the set of derangements.  The following exact values
agree with the published table in \cite{StonesEtAlComputing}.  The final
column is included only as a scale check for \eqref{eq:main-tilde}.
\begin{center}
\small
\begin{tabular}{@{}r r r@{}}
\toprule
$n$ & $L_{3,n}$ & $n^2\log(L_{3,n}/\widetilde A_{3,n})$ \\
\midrule
3 & 12 & 3.346744 \\
4 & 576 & 1.257544 \\
5 & 66240 & -1.248072 \\
6 & 15321600 & -0.585342 \\
7 & 5411750400 & -0.618311 \\
8 & 2834466324480 & -0.596572 \\
\bottomrule
\end{tabular}
\end{center}
These values can be verified directly from \eqref{eq:k3-exact-check} by
enumerating derangements and evaluating each permanent with Ryser's formula.

\section*{Acknowledgements}

The author acknowledges the use of OpenAI's ChatGPT during the preparation of
this manuscript.  While it was used for ideation, formulation, proof
exploration and refinement, narrowing the search space, programming, \LaTeX{}
formatting and other forms of orchestration, the author nonetheless takes full
responsibility for the accuracy of the final contents of this paper.

The accompanying Lean formalisation \cite{LiLean725} was developed by the
author with Codex; the author likewise takes full responsibility for the
formal statements' fidelity to the results of this paper.

\end{document}